\documentclass{amsart}
\usepackage{amsmath, amssymb}
\usepackage[colorlinks=true]{hyperref}

\theoremstyle{plain}
 \newtheorem{theorem}{Theorem}[section]
 \newtheorem{lemma}{Lemma}[section]
 \newtheorem{proposition}{Proposition}[section]
 \newtheorem{corollary}{Corollary}[section]

\theoremstyle{definition}
 \newtheorem{remark}[theorem]{Remark}

\def\mcalL{\mathcal{L}}
\def\mcalM{\mathcal{M}}

\def\mcalS{\mathcal{S}}
\def\mcalF{\mathcal{F}}

\def\mcalG{\mathcal{G}}

\def\mcalO{\mathcal{O}}

\def\mcalT{\mathcal{T}}

\def\mbbC{\mathbb{C}}
\def\mbbR{\mathbb{R}}

\def\mbbZ{\mathbb{Z}}
\def\mbbQ{\mathbb{Q}}

\def\mbbP{\mathbb{P}}

\def\Gal{\mathrm{Gal}}

\DeclareMathOperator{\Pic}{Pic}

\DeclareMathOperator{\Mor}{Mor}

\DeclareMathOperator{\End}{End}

\DeclareMathOperator{\Hom}{Hom}
\DeclareMathOperator{\identity}{id}

\DeclareMathOperator{\Impart}{Im}
\DeclareMathOperator{\Repart}{Re}
\DeclareMathOperator{\Trace}{Tr}

\DeclareMathOperator{\vol}{vol}
\DeclareMathOperator{\rank}{rank}
\DeclareMathOperator{\Ker}{Ker}
\DeclareMathOperator{\diag}{diag}

\DeclareMathOperator{\can}{can}

\title{On the minimal degree of morphisms between algebraic curves}

\author{Roland Paulin}
\address{Roland Paulin, Department of Mathematics, University of Salzburg, Hellbrunnerstr.\ 34/I, 5020 Salzburg, Austria}
\email{paulinroland@gmail.com}
\thanks{The author was supported by the Austrian Science Fund (FWF): P24574.}

\subjclass[2010]{Primary 11G30; Secondary 11G05, 11G10, 14H40, 14K02, 14K15}
\keywords{minimal degree, morphism of curves, modular degree, elliptic curve, abelian variety, isogeny estimate, Faltings height, abc conjecture}
\date{\today}

\begin{document}

\begin{abstract}
Given smooth, projective, geometrically integral algebraic curves $X$ and $Y$ defined over a number field $K$, assuming that there is a non-constant $K$-morphism $\varphi \colon X \to Y$, we give an upper bound on the minimum of the degrees of such morphisms.
The proof is based on isogeny estimates between abelian varieties.
\end{abstract}

\maketitle

%=======================================================
\section{Introduction} \label{sec:Introduction}
%=======================================================

We study the following problem.
Given two smooth projective curves defined over a number field $K$, can one bound from above the minimum of the degrees of non-constant morphisms defined over $K$ between the two curves?
The principal motivation to study this problem comes from the abc conjecture.

We first recall the modular degree conjecture, which is very closely related to the abc conjecture.
Let $E/\mbbQ$ be an elliptic curve, then the work of Wiles and others (\cite{Wiles}, \cite{BrCoDiTa}) shows that there is a non-constant $\mbbQ$-morphism $\phi$ from the modular curve $X_0(N)$ onto $E$, where $N$ is the conductor (see \cite[p.\ 256]{Sil}) of $E$.
Such a morphism $\phi$ is called a modular parametrization of $E$.
A version of the modular degree conjecture (see \cite[p.\ 178]{Murty:Bounds}) states that every elliptic curve $E/\mbbQ$ has such a modular parametrization of degree $O(N^{2+\varepsilon})$ for every $\varepsilon > 0$.
It is proved in \cite[Theorem 1]{Murty:Bounds} that this implies the abc conjecture, and conversely, the abc conjecture implies the modular degree conjecture for all Frey curves $E/\mbbQ$.
In fact, \cite[Corollary 3.1]{Frey:ternary} describes an equivalence between the abc conjecture and a slightly modified version of the modular degree conjecture.
So clearly the minimal degree is of great interest in the case of morphisms from modular curves to elliptic curves.

We give a general explicit upper bound for the minimal degree of morphisms of curves defined over a number field $K$, where, if the curves have a $K$-rational point, then the bound only depends on the genus of the curves, the degree of the number field, and the Faltings height of the Jacobians of the curves.
The proof uses isogeny estimates and endomorphism estimates for abelian varieties over number fields.
These were originally developed by Masser and W\"ustholz (see \cite{MW:IsogEst}, \cite{MW:EndoEst}, \cite{MW:FactEst}), and recently refined by Gaudron and R\'emond (see \cite{GauRe}, \cite{GauRePeriod}).

Let $X$ and $Y$ be smooth geometrically integral projective curves defined over a number field $K$, and assume that there is a non-constant $K$-morphism $X \to Y$.
Let
\[
\mu_K(X,Y) = \min \{\deg f ; \, f \colon X \to Y \textrm{ is a non-constant } K \textrm{-morphism} \}.
\]
For us the most interesting case will be when the genus of $Y$ is $g(Y) = 1$.

Before we can state our main result, we need to introduce some notions.
Let $i(X)$ denote the index of the curve $X$, which is the smallest positive integer $m$ such that there is a divisor (defined over $K$) on $X$ of degree $m$.
Let $p(X)$ denote the period of the curve $X$, which is the smallest positive integer $m$ such that there is a divisor $D$ of degree $m$ on $X_{\overline{K}}$, such that $\sigma D$ is linearly equivalent to $D$ for every $\sigma \in \Gal(\overline{K}/K)$.
Equivalently, $p(X)$ is the smallest positive integer $m$ such that there is an invertible sheaf $\mcalL$ of degree $m$ on $X_{\overline{K}}$ such that $\sigma \mcalL \cong \mcalL$ for every $\sigma \in \Gal(\overline{K}/K)$.
(Remark: In general this does not imply that $\mcalL$ is defined over $K$!)
It is easy to see that $p(X) \mid i(X)$.
For more on these notions, see e.g., \cite{LaTa}, \cite{Licht} or \cite{Cl}.

For an abelian variety $A$ of dimension $g \ge 1$ defined over the number field $k$, \cite[p.\ 2058]{GauRe} introduces the quantity
\[
\kappa(A) = \left((14g)^{64 g^2} [k:\mbbQ] \max(h_F(A), \log[k:\mbbQ], 1)^2 \right)^{2^{10} g^3},
\]
where $h_F(A)$ denotes the stable Faltings height of $A$ (with the original normalization of Faltings, see \cite[paragraph 2.3]{GauRePeriod}).
Using this notation, we can finally state our main result.

\begin{theorem} \label{thm:Faltings height}
If $g(Y) = 1$ and there is a non-constant $K$-morphism $X \to Y$, then
\[
\mu_K(X, Y) \le p(X)^2 \kappa(J(X))^3.
\]
\end{theorem}

Let us briefly discuss what happens if $g(Y) \neq 1$.
We still assume that there is a non-constant $K$-morphism from $X$ to $Y$.
If $g(Y) > 1$, then every non-constant morphism from $X$ to $Y$ has degree at most $\frac{g(X)-1}{g(Y)-1} \le g(X)-1$, so in particular $\mu_K(X,Y) \le g(X)-1$.
This immediately follows from Hurwitz's theorem, which says that
\[
2g(X)-2 = (\deg \varphi) (2g(Y)-2) + \deg R,
\]
where $R$ is the ramification divisor of $\varphi$, and $\deg R \ge 0$.

We will show now that if $g(Y) = 0$ and $X(K) \neq \varnothing$, then $\mu_K(X,Y) \le g(X)+1$.
Note that $Y(K) \neq \varnothing$ (because there is a $K$-morphism from $X$ to $Y$), so $Y \cong \mbbP^1_K$.
Let $P \in X(K)$, then $D = (g(X)+1) P$ is a divisor on $X$ of degree $g(X)+1$.
Then by the Riemann-Roch theorem, we have $\dim_K(H^0(X, \mcalL(D))) \ge \deg(D) - g(X) + 1 = 2$.
Here $\mcalL(D)$ is an invertible sheaf over $X$.
The constant $1$ function is in $H^0(X, \mcalL(D))$.
There is a non-constant section $f \in H^0(X, \mcalL(D))$ too.
Then $f$ defines a non-constant $K$-morphism $X \to \mbbP^1_K$, with degree at most $\deg(D) = g(X)+1$.

Here is a brief outline of the paper.
In section 2, we discuss some basic results about curves, abelian varieties, polarizations, endomorphisms of abelian varieties, and Rosati norms of homomorphisms.
In section 3, we express the minimal degree $\mu_K(X,Y)$ using the Rosati norm of homomorphisms of the Jacobians, reducing the problem to understanding the lattice $\Hom_K(J(X),J(Y))$.
In fact, it is enough to calculate the covolume of this lattice, which is done in section 4.
Finally, in section 5 we prove the bound of Theorem \ref{thm:Faltings height} for $\mu_K(X,Y)$.

%=======================================================
\section{Preliminaries}
%=======================================================

We introduce some notations.
Throughout this section, $K$ denotes a field of characteristic zero, and $\overline{K}$ denotes an algebraic closure of $K$.
If $A$ is an abelian variety over $K$, and $m \in \mbbZ$, then $[m]_A \colon A \to A$ denotes the multiplication by $m$ map.
The dual abelian variety of $A$ is denoted by $\widehat{A}$.
If $f \colon A \to B$ is a homomorphism of abelian varieties, then $\widehat{f} \colon \widehat{B} \to \widehat{A}$ is the dual homomorphism.

Let $A$ be an abelian variety over $K$.
For every invertible sheaf $\mcalL$ on $A$, we define a $K$-homomorphism $\phi_{\mcalL} \colon A \to \widehat{A}$, $x \mapsto \tau_x^* \mcalL \otimes \mcalL^{-1}$, where $\tau_x \colon A \to A$ is the translation by $x$.
Then $\deg(\phi_{\mcalL}) = \chi(A, \mcalL)^2$ (see \cite[Corollary 3.6.2]{BirLa}).
Thus $\phi_{\mcalL}$ is an isogeny if and only if $\mcalL$ is nondegenerate, so e.g., if $\mcalL$ is ample (see part \eqref{item:prop:basics, h^i} of Proposition \ref{prop:basics} below).

A polarization of $A$ is a $K$-isogeny $\epsilon \colon A \to \widehat{A}$ such that the base change $\epsilon_{\overline{K}}$ is of the form $\phi_{\mcalL}$ for an ample invertible sheaf $\mcalL$ on $A_{\overline{K}}$.
We say in this case that $(A, \mcalL)$ is a polarized abelian variety.
By abuse of notation, we often talk about the polarization $\mcalL$ instead of $\phi_{\mcalL}$.
(Note that $\mcalL$ is usually not uniquely determined by $\phi_{\mcalL}$.)
We say that $\mcalL$ is a principal polarization, if $\phi_{\mcalL}$ is an isomorphism.

Let $(A, \mcalL)$ and $(B, \mcalM)$ be polarized abelian varieties over $K$.
The Rosati involution
\[
\cdot^{\dag} \colon \Hom_K(A,B) \otimes_{\mbbZ} \mbbQ \to \Hom_K(B,A) \otimes_{\mbbZ} \mbbQ
\]
is defined by $f \mapsto f^{\dag} = \phi_{\mcalL}^{-1} \circ \widehat{f} \circ \phi_{\mcalM}$.
(Note that this does not depend on the exact value of $\mcalL$ and $\mcalM$, only on $\phi_{\mcalL}$ and $\phi_{\mcalM}$.)
There is a so called Rosati metric on $\Hom_K(A,B) \otimes_{\mbbZ} \mbbR$ defined in \cite{GauRe}.
This is induced by the quadratic form $q_{(A,B)} = q_{(\mcalL, \mcalM)}$, where $q_{(\mcalL,\mcalM)}(f) = \Trace(f^{\dag} \circ f) = \Trace(f \circ f^{\dag})$ for every $f \in \Hom_K(A,B)$.
(For the definition of the trace map $\Trace \colon \End_K(A) \otimes_{\mbbZ} \mbbR \to \mbbR$, see \cite[D\'efinition 2.1]{GauRe}.)
The norm of an element $f \in \Hom_K(A,B)$ is $|f| = |f|_{(\mcalL, \mcalM)} = \sqrt{q_{(\mcalL, \mcalM)}(f)}$.
The Rosati involution is clearly an isometry.
Note that if $L/K$ is a field extension, then the canonical map $\Hom_K(A,B) \otimes_{\mbbZ} \mbbR \to \Hom_L(A,B) \otimes_{\mbbZ} \mbbR$ is injective, and the Rosati metric on $\Hom_K(A,B) \otimes_{\mbbZ} \mbbR$ is the restriction of the Rosati metric on $\Hom_L(A,B) \otimes_{\mbbZ} \mbbR$.

The following proposition collects together some basic results about invertible sheaves on curves and on abelian varieties.

\begin{proposition} \label{prop:basics}
\begin{enumerate}
\item \label{item:prop:basics, deg=h^0}
If $E$ is an elliptic curve over $K$, and $\mcalL$ is an ample invertible sheaf on $E$, then $\deg \mcalL = h^0(E, \mcalL)$.

\item \label{item:prop:basics, deg(f^*L)}
If $X$ and $Y$ are integral smooth projective curves over $K$, $f \colon X \to Y$ is a non-constant $K$-morphism, and $\mcalL$ is an invertible sheaf on $Y$, then $\deg(f^* \mcalL) = \deg f \cdot \deg \mcalL$.

\item \label{item:prop:basics, isogeny}
If $f \colon A \to B$ is an isogeny of abelian varieties, and $\mcalL$ is an invertible sheaf on $B$, then $\chi(A, f^* \mcalL) = \deg f \cdot \chi(B, \mcalL)$.

\item \label{item:prop:basics, h^i}
If $A$ is an abelian variety over $K$, and $\mcalL$ is an ample invertible sheaf on $A$, then $\chi(A, \mcalL) = h^0(A, \mcalL) > 0$ and $h^i(A, \mcalL) = 0$ for every $i > 0$.

\item \label{item:prop:basics, finite ker}
If $f \colon A \to B$ is a homomorphism of abelian varieties such that $\Ker f$ is finite, and $\mcalL$ is an ample invertible sheaf on $B$, then $f^* \mcalL$ is an ample invertible sheaf on $A$.
\end{enumerate}
\end{proposition}
\begin{proof}
\eqref{item:prop:basics, deg=h^0}:
See Proposition 5.5, Lemma 3.30 and Example 3.35 in \cite[Chapter 7]{Liu}.

\eqref{item:prop:basics, deg(f^*L)}:
This follows from Proposition 3.8 of \cite[Chapter 7]{Liu}.

\eqref{item:prop:basics, isogeny}:
See \cite[Corollary 3.6.6]{BirLa} or \cite[(9.12)]{EMvdG}.

\eqref{item:prop:basics, h^i}:
See \cite[p.\ 150]{Mum} or \cite[Chapter 9]{EMvdG}.

\eqref{item:prop:basics, finite ker}:
Let $A'$ be the image of $f$, then $f$ decomposes as $f = \iota_{A'} \circ g$, where $g \colon A \to A'$ is an isogeny and $\iota_{A'} \colon A' \to B$ is a closed immersion.
Both closed immersions and isogenies of abelian varieties are finite morphisms (see \cite[(5.2)]{EMvdG}), so $f$ is also finite.
Every finite morphism is affine (so in particular quasi-affine), hence $f^* \mcalL$ is ample by \cite[Proposition 5.1.12]{EGAII}.
\end{proof}

\begin{lemma} \label{lemma:phi_f^*F = f' phi f}
If $f \colon A \to B$ is a homomorphism of abelian varieties defined over $K$, and $\mcalF$ is an invertible sheaf on $B$, then $\phi_{f^* \mcalF} = \widehat{f} \circ \phi_{\mcalF} \circ f$.
\end{lemma}
\begin{proof}
Both $\phi_{f^* \mcalF}$ and $\widehat{f} \circ \phi_{\mcalF} \circ f$ are $K$-homomorphisms $A \to \widehat{A}$, and if $x \in A$, then
\[
\phi_{f^* \mcalF}(x) = \tau_x^* f^* \mcalF \otimes (f^* \mcalF)^{-1} = (f \circ \tau_x)^* \mcalF \otimes f^* \mcalF^{-1}
\]
and
\[
(\widehat{f} \circ \phi_{\mcalF} \circ f)(x) = f^* (\tau_{f(x)}^* \mcalF \otimes \mcalF^{-1}) =  (\tau_{f(x)} \circ f)^* \mcalF \otimes f^* \mcalF^{-1}.
\]
The lemma follows from the identity $f \circ \tau_x = \tau_{f(x)} \circ f$, which is true, because $f \colon A \to B$ is a homomorphism.
\end{proof}

If $A$ is an abelian variety over $K$, then there is a canonical invertible sheaf on $A \times \widehat{A}$, called the Poincar\'e invertible sheaf (see \cite[p.\ 78]{Mum} or \cite[Section 2.5]{BirLa}).
Using the Poincar\'e invertible sheave one can define a canonical isomorphism $\can_A \colon A \to \widehat{\widehat{A}}$ (see \cite[II.8, Proposition 2]{Mum} or \cite[Chapter VII, (7.7) and (7.9)]{EMvdG}).

The following proposition describes the polarizations of an elliptic curve.

\begin{proposition} \label{prop:basics2}
Let $E$ be an elliptic curve over $K$.
Then $E$ has a unique principal polarization $\epsilon_E = \phi_{\mcalL([0_E])} \colon E \to \widehat{E}$, where $\mcalL([0_E])$ is the invertible sheaf on $E$ associated to the divisor $[0_E]$.
Furthermore, every polarization of $E$ has the form $m \epsilon_E$ for some positive integer $m$.
Moreover $\widehat{\epsilon_E} = (\epsilon_{\widehat{E}})^{-1}$ and $\epsilon_{\widehat{E}} \circ \epsilon_E = \can_E$.
\end{proposition}
\begin{proof}
Let $\mcalG$ be an invertible sheaf on $E_{\overline{K}}$.
The genus of $E$ is $1 = 1 - \chi(E, \mcalO_E)$, so $\chi(E, \mcalO_E) = 0$, hence $\deg \mcalG = \chi(E, \mcalG) - \chi(E, \mcalO_E) = \chi(E, \mcalG)$ (see Definition 3.29 in \cite[Chapter 7]{Liu}).
Using $\deg \phi_{\mcalG} = \chi(E, \mcalG)^2 = (\deg \mcalG)^2$, we see that $\phi_{\mcalG} = 0$ if and only if $\deg \mcalG = 0$.

Let $\mcalF$ be an ample invertible sheaf on $E_{\overline{K}}$.
Then $m = \deg \mcalF$ is a positive integer (see Proposition 5.5 in \cite[Chapter 7]{Liu}), and $\deg \mcalG = 0$ for $\mcalG = \mcalF \otimes \mcalL([0_E])^{-m}$.
Hence $0 = \phi_{\mcalG} = \phi_{\mcalF} - m \phi_{\mcalL[0_E]} = \phi_{\mcalF} - m \epsilon_E$, so $\phi_{\mcalF} = m \epsilon_E$.
So every polarization of $E$ has the form $m \epsilon_E$ for some positive integer $m$.
Conversely, if $m$ is a positive integer, then $m \epsilon_E = \phi_{\mcalL(m [0]_E)}$ is a polarization, because $\mcalL(m [0]_E)$ is an ample invertible sheaf on $E$.
Furthermore $\deg(m \epsilon_E) = \deg(\phi_{\mcalL(m [0_E])}) = (\deg \mcalL(m [0_E]))^2 = m^2$, so $\epsilon_E$ is indeed the unique principal polarization of $E$.

Note that $\epsilon_E^*(\mcalL(0_{\widehat{E}})) = \mcalL(0_E)$, because $\epsilon_E$ is an isomorphism and $\epsilon_E(0_E) = 0_{\widehat{E}}$.
So using Lemma \ref{lemma:phi_f^*F = f' phi f} for the homomorphism $\epsilon_E \colon E \to \widehat{E}$ and the invertible sheaf $\mcalL(0_{\widehat{E}})$ on $\widehat{E}$, we obtain $\epsilon_E = \widehat{\epsilon_E} \circ \epsilon_{\widehat{E}} \circ \epsilon_E$, which implies $\widehat{\epsilon_E} = (\epsilon_{\widehat{E}})^{-1}$.

Using $\widehat{\epsilon_E} = (\epsilon_{\widehat{E}})^{-1}$ we see that the last identity of the proposition is equivalent to $\epsilon_E = \widehat{\epsilon_E} \circ \can_E$.
More generally $\phi_{\mcalF} = \widehat{\phi_{\mcalF}} \circ \can_E$ is true for every invertible sheaf $\mcalF$ on $E$ (see \cite[Corollary 2.4.6 (c)]{BirLa} or \cite[Chapter VII, (7.8)]{EMvdG}).
\end{proof}

The following Proposition describes the Rosati norm of endomorphisms of elliptic curves.

\begin{proposition} \label{prop:endoms of elliptic curves}
Let $(E, \mcalM)$ be an elliptic curve over $K$ with its canonical principal polarization.
Then $\mcalO = \End_K(E)$ is either $\mbbZ$ or an order in an imaginary quadratic number field.
Let $c^{\dag} = \phi_{\mcalM}^{-1} \circ \widehat{c} \circ \phi_{\mcalM} \in \mcalO$ for every $c \in \mcalO$.
Then $c^{\dag}$ is the complex conjugate of $c$, $\Trace c = c + c^{\dag} = 2 \Repart c$, and $\deg c = c c^{\dag} = |c|^2$.
The Rosati norm of $c \in \mcalO$ is $|c|_{(\mcalM, \mcalM)} = \sqrt{2} |c|$.
(Here $|c|$ denotes the complex absolute value of $c$.)
\end{proposition}
\begin{proof}
Note that $\mcalO$ is a subring of $\End_{\mbbC}(E)$, so the first part follows from \cite[Ch.\ VI, Theorem 5.5]{Sil}.
The complex conjugation and the Rosati involution are both automorphisms of the ring $\mcalO$, and they are both involutions.
If $\mcalO = \mbbZ$, then the only such involution is $\identity_{\mbbZ}$.
So suppose that $\mcalO = \mbbZ + \mbbZ \tau$, where $\tau \in \mbbC$ is an imaginary quadratic integer.
Then $\tau$ is a root of a polynomial $X^2 + uX + v \in \mbbZ[X]$.
Here $X^2 + uX + v = (X-\tau)(X-\overline{\tau})$, and $(\tau^{\dag})^2 + u \tau^{\dag} + v = 0$, so $\tau^{\dag} \in \{\tau, \overline{\tau}\}$.
Thus the Rosati involution is either the identity or the complex conjugation on $\mcalO$.
Suppose indirectly that it is the identity.
There is an element $s \in \mcalO$ such that $s^2 = -d$ for some positive integer $d$.
Then
\[
0 \le |s|_{(\mcalM, \mcalM)}^2 = \Trace(s s^{\dag}) = \Trace(-d) = -2d < 0,
\]
contradiction (here we have used \cite[Propri\'et\'e 2.2 (4)]{GauRe}).
So $c^{\dag} = \overline{c}$ for every $c \in \mcalO$.

Using $\Trace(c^{\dag}) = \Trace c$ and $\deg(c^{\dag}) = \deg c$ (see the paragraphs after D\'efinition 2.4 in \cite{GauRe}), we obtain that $2 \Trace c = \Trace(c + c^{\dag}) = \Trace(2 \Repart c) = 4 \Repart c$ and $(\deg c)^2 = \deg(c c^{\dag}) = \deg(|c|^2) = |c|^4$.
Thus $\Trace c = 2 \Repart c = c + c^{\dag}$ and $\deg c = |c|^2 = c c^{\dag}$.
Finally, $|c|_{(\mcalM, \mcalM)}^2 = \Trace(c c^{\dag}) = \Trace(|c|^2) = 2 |c|^2$.
\end{proof}

The following two lemmas describe some basic properties of the Rosati norm.

\begin{lemma} \label{lemma:|gf| and |f|}
Let $(A, \mcalL)$ and $(C, \mcalM)$ be polarized abelian varieties over $K$, let $B$ be an abelian variety over $K$, and let $f \in \Hom_K(A,B)$ and $g \in \Hom_K(B,C)$.
If $\Ker g$ is finite, then $g^* \mcalM$ is a polarization of $B$, and $|g \circ f|_{(\mcalL,\mcalM)} = |f|_{(\mcalL,g^* \mcalM)}$.
\end{lemma}
\begin{proof}
Let $\Ker g$ be finite.
Then $g$ is finite, hence it is affine and therefore quasi-affine, so $g^* \mcalM$ is a polarization of $B$ (i.e., an ample invertible sheaf on $B$) by Proposition 5.1.12 of \cite{EGAII}.
Using the definition of the Rosati norm, we obtain
\begin{align*}
|g \circ f|_{(\mcalL,\mcalM)}^2 &= \Trace((g \circ f)^{\dag} \circ g \circ f) = \Trace(\phi_{\mcalL}^{-1} \circ \widehat{f} \circ \widehat{g} \circ \phi_{\mcalM} \circ g \circ f) \\
&= \Trace(\phi_{\mcalL}^{-1} \circ \widehat{f} \circ \phi_{g^* \mcalM} \circ f) = \Trace(f^{\dag} \circ f) = |f|_{(\mcalL,g^* \mcalM)}^2,
\end{align*}
because $\phi_{g^* \mcalM} = \widehat{g} \circ \phi_{\mcalM} \circ g$ by Lemma \ref{lemma:phi_f^*F = f' phi f}.
\end{proof}

\begin{lemma} \label{lemma:|f|^2=2h^0(f^*F)}
Let $(A, \mcalF)$ be a polarized abelian variety over $K$, and let $(E, \mcalM)$ be an elliptic curve over $K$, with its canonical principal polarization.
If $f \in \Hom_K(E,A)$, then
\[
|f|_{(\mcalM, \mcalF)}^2 = 2 \chi(E, f^* \mcalF).
\]
\end{lemma}
\begin{proof}
This is trivial for $f = 0$, so suppose $f \neq 0$.
Then $\Ker f$ is finite, so $\mcalG = f^* \mcalF$ is a polarization of $E$ and $|f|_{(\mcalM, \mcalF)} = |\identity_E|_{(\mcalM, \mcalG)}$ by Lemma \ref{lemma:|gf| and |f|}.
By Proposition \ref{prop:basics2}, there is a positive integer $m$ such that $\phi_{\mcalG} = \phi_{\mcalM} \circ [m]_E$.
Then $m^2 = \deg(\phi_{\mcalG}) = \chi(E, \mcalG)^2$, where $\chi(E, \mcalG) > 0$ by part \eqref{item:prop:basics, h^i} of Proposition \ref{prop:basics}, so $m = \chi(E, \mcalG)$.
Thus by \cite[Propri\'et\'e 2.2 (4)]{GauRe},
\[
|f|^2_{(\mcalM, \mcalF)} = |\identity_E|_{(\mcalM, \mcalG)}^2 = \Trace(\phi_{\mcalM}^{-1} \circ \phi_{\mcalG}) = \Trace([m]_E) = 2m = 2 \chi(E, \mcalG).
\]
\end{proof}

The following lemma gives an expression for the degree of an endomorphism of $E^r$, where $r$ is a positive integer and $E$ is an elliptic curve.
We will need this to prove Lemma \ref{lemma:matrix of phi_F}.

\begin{lemma} \label{lemma:deg phi = det M^2}
Let $E$ be an elliptic curve over $K$, $\mcalO = \End_K(E)$, $r$ a positive integer, $M \in \mcalO^{r \times r}$, and $\phi \in \End_K(E^r) \cong \mcalO^{r \times r}$ the endomorphism corresponding to $M$.
Then $\deg \phi = |\det M|^2$.
\end{lemma}
\begin{proof}
Let $I \in \mcalO^{r \times r}$ denote the identity matrix, and $E_{i,j} \in \mcalO^{r \times r}$ the matrix whose $(i,j)$ entry is $1$ and all other entries are zero.
If $M \in \mcalO^{r \times r}$, then let $d(M) = \deg \phi$, where $\phi \in \End_K(E^r) \cong \mcalO^{r \times r}$ is the endomorphism corresponding to $M$.
Note that $d(I) = 1$ and $d(MN) = d(M) d(N)$ for every $M,N \in \mcalO^{r \times r}$.
Let
\[
\mcalT = \{M \in \mcalO^{r \times r}; d(M) = |\det M|^2 \}.
\]
The following statements hold for every $M, N \in \mcalO^{r \times r}$.
\begin{itemize}
\item If $M,N \in \mcalT$, then $MN \in \mcalT$.
\item If $M, MN \in \mcalT$ and $\det M \neq 0$, then $N \in \mcalT$.
\item If $N, MN \in \mcalT$ and $\det N \neq 0$, then $M \in \mcalT$.
\end{itemize}

If $M = \diag(c_1, \dotsc, c_r)$, then $\deg c_i = |c_i|^2$ for every $i \in \{1, \dotsc, r\}$ by Proposition \ref{prop:endoms of elliptic curves}, so $d(M) = \prod_{i=1}^r \deg c_i = \prod_{i=1}^r |c_i|^2 = |\det M|^2$.
Thus every diagonal matrix is in $\mcalT$.
If $M \in \mcalO^{r \times r}$ and $\det M = \pm 1$, then $M^{-1} \in \mcalO^{r \times r}$, so $1 = d(I) = d(M M^{-1}) = d(M) d(M^{-1})$, hence $d(M) = 1 = |\det M|^2$.
So every permutation matrix and every matrix of the form $I + \lambda E_{i,j}$ (where $i,j \in \{1, \dotsc, r\}$ and $\lambda \in \mcalO$) is in $\mcalT$.
Let $\mcalS$ the set of matrices in $\mcalO^{r \times r}$ which can be written as finite products $U_1 \dotsm U_k$, where $k \ge 0$, and each $U_l \in \mcalO^{r \times r}$ is either a permutation matrix or a matrix of the form $I + \lambda E_{i,j}$, and $\det(U_l) \neq 0$ (so if $i=j$, then $\lambda \neq -1$).
Then $\mcalS \subseteq \mcalT$, and if $U,V \in \mcalS$, then $M \in \mcalT$ if and only if $UMV \in \mcalT$.
One can easily check that for every $M \in \mcalO^{r \times r}$ there are $U,V \in \mcalS$ such that $UMV$ is diagonal.
Then $UMV \in \mcalT$, hence $M \in \mcalT$.
\end{proof}

\begin{lemma} \label{lemma:matrix of phi_F}
Let $(E, \mcalM)$ be an elliptic curve over $K$, with its canonical principal polarization, and let $r$ be a positive integer.
Using the isomorphism $\phi_{\mcalM}$ we identify $\widehat{E}$ with $E$, and we also identify $\widehat{E^r}$, $\widehat{E}^r$ and $E^r$.
Let $\mcalO = \End_K(E)$, then we have a canonical ring isomorphism $\End_K(E^r) \cong \mcalO^{r \times r}$.
If $\mcalF$ is a polarization of $E^r$, then $\phi_{\mcalF} \in \Hom_K(E^r, \widehat{E^r}) = \End_K(E^r)$ corresponds to a hermitian positive definite matrix $M \in \mcalO^{r \times r}$, and $\det M = \chi(E^r, \mcalF)$.
\end{lemma}
\begin{proof}
The identity $\widehat{\phi_{\mcalF}} = \phi_{\mcalF}$ (see \cite[(7.8)]{EMvdG} or \cite[Corollary 2.4.6 (c)]{BirLa}) implies that $M$ is hermitian.
First let $x \in \mcalO^r \setminus \{0\}$, and let $\alpha \colon E \to E^r$ be the corresponding nonzero homomorphism.
Then $x^* M x = \widehat{\alpha} \phi_{\mcalF} \alpha = \phi_{\alpha^* \mcalF}$ by Lemma \ref{lemma:phi_f^*F = f' phi f}.
Here $\alpha^* \mcalF$ is an ample invertible sheaf on $E$ by Lemma \ref{lemma:|gf| and |f|} (since $\Ker \alpha$ is finite), so $\phi_{\alpha^* \mcalF} = [n]_E$ for some positive integer $n$, hence $x^* M x = n > 0$.
Let $F$ denote the quotient field of $\mcalO$.
Then $F = \mbbQ$ if $\mcalO = \mbbZ$, while otherwise $F$ is an imaginary quadratic number field.
If $x \in F^r$, then there is a positive integer $k$ such that $k x \in \mcalO^r$, so $k^2 (x^* M x) = (kx)^* M (kx) \ge 0$, hence $x^* M x \ge 0$.
If $\mcalO=\mbbZ$, then $F = \mbbQ$ is dense in $\mbbR$, while otherwise $F$ is dense in $\mbbC$, so using the continuity of the map $x \mapsto x^* M x$, we obtain that $M$ is positive semidefinite.
Then $\det M \ge 0$, so it is now enough to prove that $\det M = \chi(E^r, \mcalF) > 0$.

According to part \eqref{item:prop:basics, h^i} of Proposition \ref{prop:basics}, $\chi(E^r, \mcalF) > 0$, because $\mcalF$ is ample.
Using Lemma \ref{lemma:deg phi = det M^2}, we obtain
\[
\chi(E^r, \mcalF)^2 = \deg(\phi_{\mcalF}) = (\det M)^2.
\]
Here $\det M \ge 0$, so $\det M = \chi(E^r, \mcalF) > 0$.
\end{proof}

We introduce a few notions and notations for lattices.
Let $\Omega \neq \{0\}$ be a lattice in a real euclidean space $(E, |\cdot|)$ of dimension $n$.
Let
\[
\lambda(\Omega) = \min \{|x| ; \, x \in \Omega \setminus \{0\}\}.
\]
Let $\Lambda(\Omega)$ be the smallest $\lambda \ge 0$ such that there are linearly independent $\omega_1, \dotsc, \omega_n \in \Omega$ with $|\omega_1|, \dotsc, |\omega_n| \le \lambda$.
Equivalently, $\Lambda(\Omega)$ is the smallest $\lambda \ge 0$ such that $\{\omega \in \Omega; \, |\omega| \le \lambda\}$ generates a finite index subgroup of $\Omega$.
The covering radius $R(\Omega)$ of $\Omega$ is the smallest $\rho \ge 0$ such that for every $x \in E$, there is an $\omega \in \Omega$ with $|x-\omega| \le \rho$.
Let $\vol(\Omega)$ denote the covolume of $\Omega$.
One can easily check that $R(\Omega) \le \frac{1}{2} \sqrt{n} \Lambda(\Omega)$ and $\lambda(\Omega) \le \Lambda(\Omega) \le 2R(\Omega)$.

In the following elementary lemma we construct a matrix $S \in \mbbR^{2r \times 2r}$ from a matrix $M \in \mbbC^{r \times r}$ and a number $\tau \in \mbbC$, and get an expression for $\det S$ using $\det M$ and $\Impart \tau$.
We will use this lemma for covolume calculations of lattices in section \ref{section:covolume-calculation}.

\begin{lemma} \label{lemma:matrix lemma}
Let $r$ be a positive integer, $A, B \in \mbbC^{r \times r}$, $s,t \in \mbbC$, and $N = \left(\begin{smallmatrix} A+B & sA + tB \\ tA+sB & st(A+B) \end{smallmatrix}\right) \in \mbbC^{2r \times 2r}$.
Then
\[
\det N = (-1)^r (t-s)^{2r} (\det A) (\det B).
\]
If $M \in \mbbC^{r \times r}$, $\tau \in \mbbC$ and $S = \left(\begin{smallmatrix} \Repart(M) & \Repart(\tau M) \\ \Repart(\overline{\tau} M) & |\tau|^2 \Repart(M) \end{smallmatrix}\right) \in \mbbR^{2r \times 2r}$, then
\[
\det S = (\Impart \tau)^{2r} |\det M|^2.
\]
\end{lemma}
\begin{proof}
Using row and column operations, we see that
\begin{align*}
\det N &= \det \left(\begin{smallmatrix} A+B & sA + tB \\ (t-s)A & s(t-s)A \end{smallmatrix}\right) = (t-s)^r \det \left(\begin{smallmatrix} A+B & sA + tB \\ A & sA \end{smallmatrix}\right) \\
&= (t-s)^r \det \left(\begin{smallmatrix} A+B & (t-s)B \\ A & 0 \end{smallmatrix}\right) = (-1)^r (t-s)^{2r} (\det A) (\det B).
\end{align*}
Substituting $A = \frac{1}{2} M$, $B = \frac{1}{2} \overline{M}$, $s = \tau$ and $t = \overline{\tau}$, we obtain $N = S$ and
\[
\det S = (-1)^r (\overline{\tau} - \tau)^{2r} \det\left(\frac{1}{2} M\right) \det\left(\frac{1}{2} \overline{M}\right) = (\Impart \tau)^{2r} |\det M|^2.
\]
\end{proof}

%=======================================================
\section{Morphisms and Jacobians}
%=======================================================

We fix the following notation for this section.
Let $K$ denote a field of characteristic zero, and let $\overline{K}$ be an algebraic closure of $K$.
Let $X$ and $Y$ be geometrically integral smooth projective curves defined over $K$, of genus $g(X)$ and $g(Y)$.
Let $\Omega = \Hom_K(J(X), J(Y))$, then $\Omega$ is a free abelian group of finite rank.
The vector space $\mbbR \otimes_{\mbbZ} \Omega$ together with the Rosati metric is a euclidean space, and $\Omega$ is a lattice in it.

The Jacobian of the curve $X$ is $J(X) = \Pic^0(X)$.
This is an abelian variety over $K$ of dimension $g(X)$, with a canonical principal polarization.
We have a canonical isomorphism between $J(X)(\overline{K})$ and the group of isomorphism classes of invertible sheaves on $X_{\overline{K}}$, therefore by abuse of notation we simply identify these two sets.
The canonical principal polarization $\phi \colon J(X) \to \widehat{J(X)}$ can be described as follows.
Choose any $P_0 \in X(K_0)$, then we get a $\overline{K}$-morphism $\iota_{X, P_0} \colon X \to J(X)$, where $\iota_{X, P_0}(P) = \mcalL([P]-[P_0])$ for every $P \in X(\overline{K})$.
Then $\iota_{X, P_0}^* \colon \widehat{J(X)} \to J(X)$ is a $\overline{K}$-morphism.
This does not depend on the choice of $P_0$, so it is in fact a $K$-morphism, and moreover it is an isomorphism, with inverse $-\phi$.
(See Lemma 11.3.1 and Proposition 11.3.5 in \cite{BirLa}.)

Using the principal polarizations of $J(X)$ and $J(Y)$, we can define the homomorphisms
\[
\cdot^{\dag} \colon \Hom_K(J(X), J(Y)) \to \Hom_K(J(Y), J(X))
\]
and
\[
\cdot^{\dag} \colon \Hom_K(J(Y), J(X)) \to \Hom_K(J(X), J(Y))
\]
the same way as in the Preliminaries.
Here $(f^{\dag})^{\dag} = f$ if $f \in \Hom_K(J(X), J(Y))$ or $f \in \Hom_K(J(Y), J(X))$.

If $f \colon X \to Y$ is a $K$-morphism, then there is a unique homomorphism $J(f) \in \Hom_K(J(X), J(Y))$ such that $J(f)(\mcalL) = f_* \mcalL$ for every $\mcalL \in J(X)(\overline{K})$.
Similarly, there is a unique homomorphism $J'(f) \in \Hom_K(J(Y), J(X))$ such that $J'(f)(\mcalL) = f^* \mcalL$ for every $\mcalL \in J(Y)(\overline{K})$.

\begin{lemma} \label{lemma:J'(f) = J(f)^dag}
$J'(f) = J(f)^{\dag}$.
\end{lemma}
\begin{proof}
Choose a point $P_0 \in X(\overline{K})$, and let $Q_0 = f(P_0) \in Y(\overline{K})$.
As we have seen above, the principal polarizations are $\phi_1 = -(\iota_{X,P_0}^*)^{-1} \colon J(X) \to \widehat{J(X)}$ and $\phi_2 = -(\iota_{Y,Q_0}^*)^{-1} \colon J(Y) \to \widehat{J(Y)}$.
Then $J(f)^{\dag} = \phi_1^{-1} \circ \widehat{J(f)} \circ \phi_2$.
Let $\mcalM \in J(Y)(\overline{K})$, then
\[
J(f)^{\dag}(\mcalM) = \iota_{X,P_0}^*(J(f)^*((\iota_{Y,Q_0}^*)^{-1}(\mcalM))) = (J(f) \circ \iota_{X,P_0})^*((\iota_{Y,Q_0}^*)^{-1}(\mcalM)).
\]
If $P \in X(\overline{K})$, then
\[
(J(f) \circ \iota_{X,P_0})(P) = f_*(\mcalL([P]-[P_0])) = \mcalL([f(P)] - [Q_0]) = (\iota_{Y,Q_0} \circ f)(P)),
\]
so $J(f) \circ \iota_{X,P_0} = \iota_{Y,Q_0} \circ f$.
Hence
\[
J(f)^{\dag}(\mcalM) = f^*(\iota_{Y,Q_0}^*((\iota_{Y,Q_0}^*)^{-1}(\mcalM))) = f^*(\mcalM) = J'(f)(\mcalM).
\]
\end{proof}

Let $\Mor_K(X,Y)$ denote the set of $K$-morphisms from $X$ to $Y$.
If $\varphi \colon X \to Y$ is a constant morphism, then we define its degree to be zero.

\begin{proposition} \label{prop:deg phi = 1/2g q(J(phi))}
If $\varphi \colon X \to Y$ is a $K$-morphism, and $g(Y) > 0$, then
\[
\deg \varphi = \frac{1}{2 g(Y)} |J(\varphi)|^2 = \frac{1}{2 g(Y)} |J(\varphi)^{\dag}|^2,
\]
where the Rosati norms are defined with respect to the canonical principal polarizations of $J(X)$ and $J(Y)$.
\end{proposition}
\begin{proof}
If $f \in \Omega$, then $|f|^2 = \Trace(f \circ f^{\dag}) = \Trace(f^{\dag} \circ f) = |f^{\dag}|^2$, so $|f| = |f|^{\dag}$.
As a special case we obtain $|J(\varphi)| = |J(\varphi)^{\dag}|$.
If $\varphi$ is constant, then $\deg \varphi = 0$ and $J(\varphi) = 0$, so we may assume that $\varphi$ is not constant.
Then
\[
(J(\varphi) \circ J(\varphi)^{\dag}) (\mcalL) = \varphi_*(\varphi^* \mcalL) = \mcalL^{\otimes \deg \varphi} = [\deg \varphi]_{J(Y)} (\mcalL)
\]
for every $\mcalL \in J(Y)(\overline{K})$ by Theorem 2.18 of \cite[Chapter 7]{Liu}, so
$J(\varphi) \circ J(\varphi)^{\dag} = [\deg \varphi]_{J(Y)}$.
Thus
\[
|J(\varphi)|^2 = q_{J(X),J(Y)}(J(\varphi)) = \Trace(J(\varphi) \circ J(\varphi)^{\dag}) = \Trace([\deg \varphi]_{J(Y)}) = 2g(Y) \deg \varphi
\]
by \cite[Propri\'et\'e 2.2 (4)]{GauRe}.
\end{proof}

\begin{remark}
Here is another proof for Proposition \ref{prop:deg phi = 1/2g q(J(phi))}.
Let $\mcalL$ and $\mcalM$ denote the principal polarizations of $J(X)$ and $J(Y)$.
Let us fix a point $x_0 \in J(X)(\overline{K})$, and define $y_0 = \varphi(x_0) \in J(Y)(\overline{K})$.
Using $x_0$ and $y_0$ we can construct canonical $\overline{K}$-morphisms $\iota_{X,x_0} \colon X \to J(X)$ and $\iota_{Y,y_0} \colon Y \to J(Y)$, where $\iota_X(x) = \mcalL([x]-[x_0])$ and $\iota_Y = \mcalL([y]-[y_0])$ for every $x \in X(\overline{K})$ and $y \in Y(\overline{K})$.
Then $\iota_X$ and $\iota_Y$ are closed immersions, and $\iota_Y \circ \varphi = J(\varphi) \circ \iota_X$.
According to \cite[Proposition 2.8]{GauRe}, we have
\[
|J(\varphi)|^2 = q_{(\mcalL, \mcalM)}(J(\varphi)) = \frac{2 g(X)}{\mcalL^{\cdot g(X)}} (\mcalL^{\cdot g(X)-1} \cdot J(\varphi)^* \mcalM),
\]
where the right hand side contains intersection products.
In the Chow ring $\mcalL^{\cdot g(X)} = g(X)!$ and $\mcalL^{\cdot g(X)-1} = (g(X)-1)! X$ (see \cite[Theorem 3.6.3 and Formula 11.2.1]{BirLa}), so
\[
\frac{1}{2} q_{(\mcalL, \mcalM)}(J(\varphi)) = X \cdot J(\varphi)^* \mcalM.
\]
Using the definition of the intersection product (see \cite[section 2.3]{Ful}), part \eqref{item:prop:basics, deg(f^*L)} of Proposition \ref{prop:basics}, and $Y \cdot \mcalM = g(Y)$ (see \cite[Corollary 11.2.2]{BirLa}), we obtain
\begin{align*}
\frac{1}{2} q_{(\mcalL, \mcalM)}(J(\varphi)) &= \deg(\iota_X^* J(\varphi)^* \mcalM) = \deg(\varphi^* \iota_Y^* \mcalM) = (\deg \varphi) \cdot (Y \cdot \mcalM) \\
&= g(Y) \deg \varphi.
\end{align*}
\end{remark}

From now on let $g(Y) = 1$, then $E = J(Y)$ is an elliptic curve over $K$, and $Y$ is a torsor over $E$.
According to Proposition \ref{prop:deg phi = 1/2g q(J(phi))}, $\deg \varphi = \frac{1}{2} |J(\varphi)|^2$ for every morphism $\varphi \in \Mor_K(X,Y)$.
In particular $\varphi$ is non-constant if and only if $J(\varphi) \neq 0$.
So determining the minimal degree of non-constant morphisms in $\Mor_K(X,Y)$ is equivalent to determining the minimal norm of nonzero elements in the set
\[
U := \{J(\varphi) ; \, \varphi \in \Mor_K(X,Y) \} \subseteq \Omega.
\]
Clearly $0 \in U$ if and only if $Y(K) \neq \varnothing$, and $U \neq \varnothing$ if and only if $\Mor_K(X,Y) \neq \varnothing$.
There is a non-constant $K$-morphism from $X$ to $Y$ if and only if $U$ contains a nonzero element.
The following proposition describes the structure of $U$ if $\Mor_K(X,Y) \neq \varnothing$.

\begin{proposition} \label{prop:structure of J(phi); phi:X->Y}
Suppose that $\Mor_K(X,Y) \neq \varnothing$.
Then $U$ is a coset of a finite index subgroup $H$ of $\Omega$, and $p(X) \Omega \subseteq H$.
In particular, if $X(K) \neq \varnothing$, then $U = H = \Omega$.
\end{proposition}
\begin{proof}
Since $Y$ is a genus one curve and $E$ is its Jacobian, $Y$ is a torsor over $E$.
Hence there is a canonical $K$-morphism $\kappa \colon E \times Y \to Y$, such that $\kappa(0,y) = y$, $\kappa(a,\kappa(b,y)) = \kappa(a+b, y)$ for every $a,b \in E$ and $y \in Y$, and $\eta \colon E \times Y \to Y \times Y$, $(a,y) \mapsto (\kappa(a,y), y)$ is an isomorphism.
Let $\nu = \pi_1 \circ \eta^{-1} \colon Y \times Y \to E$, where $\pi_1 \colon E \times Y \to E$ is the canonical projection.
So basically ``$\kappa(a,y)=a+y$'' and ``$\nu(y_1, y_2) = y_1-y_2$''.
Then $\kappa(\nu(y_1, y_2),y_2) = y_1$ for every $y_1, y_2 \in Y$.
Let
\[
\rho \colon Y \times Y \times Y \to Y, \quad (y_1, y_2, y_3) \mapsto \kappa(\nu(y_2,y_3), y_1)
\]
(so basically ``$\rho(y_1, y_2, y_3) = y_1+y_2-y_3$'').
We fix a morphism $f_0 \in \Mor_K(X,Y)$, and define
\[
H = \{J(f)-J(f_0) ; \, f \in \Mor_K(X,Y)\}.
\]
We need to prove that $H$ is a subgroup of $\Omega$ such that $p(Y) \Omega \subseteq H$.
Since $\Omega \cong \mbbZ^s$ for some $s \in \mbbZ_{\ge 0}$, this will imply that $|\Omega / H| \le p(Y)^s < \infty$.

Let $f,g \in \Mor_K(X,Y)$, then we can define a $K$-morphism $h \colon X \to Y$, $x \mapsto \rho(f_0(x), f(x), g(x))$.
We claim that $J(h)= J(f_0) + J(f) - J(g)$.
To see this, define the $K$-morphism
\[
d_{f,g} \colon X \to J(Y), \quad x \mapsto \nu(f(x),g(x)) = \mcalL([f(x)] - [g(x)]).
\]
Clearly $d_{f,g} = d_{h,f_0}$.
If $\mcalM \in J(X)(\overline{K})$, then we can write $\mcalM = \mcalL(\sum_{i=1}^r n_i [P_i])$ for some $P_i \in X(\overline{K})$ and $n_i \in \mbbZ$ with $\sum_{i=1}^r n_i = 0$.
Then
\begin{align*}
(J(f)-J(g))(\mcalM) &= f_* \mcalM - g_* \mcalM = \mcalL(\sum_{i=1}^r n_i ([f(P_i)]-[g(P_i)])) = \sum_{i=1}^r n_i d_{f,g}(P_i) \\
&= \sum_{i=1}^r n_i d_{h,f_0}(P_i) = \mcalL(\sum_{i=1}^r n_i ([h(P_i)]-[f_0(P_i)])) \\
&= h_* \mcalM - (f_0)_* \mcalM = (J(h)-J(f_0))(\mcalM),
\end{align*}
so indeed $J(f)-J(g) = J(h)-J(f_0)$.
Thus $J(h) - J(f_0) = (J(f) - J(f_0)) - (J(g) - J(g_0))$, which proves that $H$ is closed under subtraction.
Since $H \neq \varnothing$, this implies that $H$ is a subgroup of $\Omega$, and that $U$ is a coset of $H$.

Let $D$ be a divisor of degree $p(X)$ on $X_{\overline{K}}$, such that $\mcalL(D) \cong \mcalL(\sigma D)$ for every $\sigma \in \Gal(\overline{K}/K)$.
Then there is a unique $K$-morphism $\iota_D \colon X \to J(X)$ such that $\iota_D(P) = \mcalL(p(X) [P] - D)$ for every $P \in X(\overline{K})$.
Let $t \in \Omega$ be arbitrary.
We need to show that there is an $f \in \Mor_K(X,Y)$ such that $J(f)-J(f_0) = p(X) t$.
We can define a $K$-morphism
\[
f \colon X \to Y, \quad x \mapsto \kappa(t(\iota_D(x)), f_0(x)).
\]
Take any $\mcalM = \mcalL(\sum_{i=1}^r n_i [P_i]) \in J(X)(\overline{K})$.
Note that $d_{f,f_0} = t \circ \iota_D$, so
\begin{align*}
(J(f)-J(f_0))(\mcalM) &= \sum_{i=1}^r n_i d_{f,f_0}(P_i) = \sum_{i=1}^r n_i t(\iota_D(P_i)) \\
&= \sum_{i=1}^r n_i t(\mcalL(p(X)[P_i]-D)) = t(\mcalL(\sum_{i=1}^r n_i (p(X)[P_i]-D))) \\
&= p(X) t(\mcalM),
\end{align*}
because $\sum_{i=1}^r n_i = 0$.
So indeed $J(f)-J(f_0) = p(X) t$, therefore $p(X) \Omega \subseteq H$.

Finally, if $X(K) \neq \varnothing$, then $p(X) = 1$, so $U = H = \Omega$.
\end{proof}

\begin{corollary} \label{cor:mu(X,Y)=1/2 min(|f|)^2}
Suppose that there is a non-constant $K$-morphism from $X$ to $Y$.
Then
\[
\frac{1}{2} \lambda(\Omega)^2 \le \mu_K(X,Y) = \frac{1}{2} (\min \{|u| ; \, u \in U \setminus \{0\}\})^2.
\]
If $Y(K) = \varnothing$, then
\[
\mu_K(X,Y) \le \frac{1}{2} p(X)^2 R(\Omega)^2.
\]
If $Y(K) \neq \varnothing$, then
\[
\mu_K(X,Y) \le \frac{1}{2} p(X)^2 \lambda(\Omega)^2.
\]
If $X(K) \neq \varnothing$, then
\[
\mu_K(X,Y) = \frac{1}{2} \lambda(\Omega)^2.
\]
\end{corollary}
\begin{proof}
The first inequality and equation follow immediately from Propositions \ref{prop:deg phi = 1/2g q(J(phi))} and \ref{prop:structure of J(phi); phi:X->Y}.
Suppose that $Y(K) = \varnothing$, and choose a $u_0 \in U$.
Then there is a $\omega \in \Omega$ such that $|\omega - u_0/p(X)| \le R(\Omega)$, so $|u| \le p(x) R(\Omega)$ for $u = u_0 - p(X) \omega$.
Here $p(X) \omega \in H$, so $u \in U$, and $u \neq 0$, because $Y(K) = \varnothing$.
Thus $\mu_K(X,Y) \le \frac{1}{2} |u|^2 \le \frac{1}{2} p(X)^2 R(\Omega)^2$.

Now suppose that $Y(K) \neq \varnothing$, so $U = H$.
Take an $\omega \in \Omega \setminus \{0\}$ such that $|\omega| = \lambda(\Omega)$, then $u = p(X) \omega \in U \setminus \{0\}$, so $\mu_K(X,Y) \le \frac{1}{2} |u|^2 = \frac{1}{2} p(X)^2 \lambda(\Omega)^2$.
Finally, if $X(K) \neq \varnothing$, then $p(X) = 1$ and $U = H = \Omega$, so $\mu_K(X,Y) = \frac{1}{2} \lambda(\Omega)^2$.
\end{proof}

To get an upper bound for $\mu_K(X,Y)$ using Corollary \ref{cor:mu(X,Y)=1/2 min(|f|)^2}, it is enough to bound $\lambda(\Omega)$ and $R(\Omega)$ from above.
These upper bounds will depend on $\vol(\Omega)$.
That is why we calculate $\vol(\Omega)$ in the next section.

%=======================================================
\section{Covolume calculation} \label{section:covolume-calculation}
%=======================================================

Let $(J(X), \mcalL_{J(X)})$ and $(E, \mcalL_E)$ be the Jacobian variety of $X$ and the elliptic curve $E = J(Y)$ with their canonical polarizations.
Let $r$ be the largest nonnegative integer such that there is a $K$-morphism $\alpha \colon E^r \to J(X)$ with finite kernel.
Clearly $0 \le r \le \dim_K(J(X)) = g(X)$.
Let $A = \alpha(E^r)$ and let $B = A^{\perp}$ be the orthogonal complement of $A$ in $J(X)$ with respect to $\mcalL_{J(X)}$ (see e.g., \cite[p.\ 125]{BirLa} for the definition of the orthogonal complement).
Let $\iota_A \colon A \to J(X)$ and $\iota_B \colon B \to J(X)$ be the canonical embeddings.

We claim that here $A$ and $B$ are uniquely determined, i.e., they do not depend on $\alpha$.
To see this, take another $K$-morphism $\alpha' \colon E^r \to J(X)$ with finite kernel.
Let $\varphi \colon A \times B \to J(X)$, $(a,b) \mapsto a+b$, then $\mu$ is a $K$-isogeny, so there is a $K$-isogeny $\psi \colon J(X) \to A \times B$ and a positive integer $N$ such that $\varphi \circ \psi = [N]_{J(X)}$ and $\psi \circ \varphi = [N]_{A \times B}$ (see \cite[Proposition 1.2.6]{BirLa} or \cite[p.\ 169]{Mum} or \cite[(5.12)]{EMvdG}).
Let $\iota_1 \colon A \to A \times B$, $\iota_2 \colon B \to A \times B$, $\pi_1 \colon A \times B \to A$ and $\pi_2 \colon A \times B \to B$ be the canonical morphisms.
The maximality of $r$ implies that $\Hom_K(E,B) = \{0\}$, and then $\Hom_K(B,E) = 0$ holds too.
Hence $\pi_B \circ \psi \circ \alpha' = 0$, therefore $\psi \circ \alpha' = \iota_1 \circ \beta$ for some $\beta \in \Hom_K(E^r,A)$.
Then
\[
[N]_{J(X)} \circ \alpha' = \varphi \circ \psi \circ \alpha' = \varphi \circ \iota_1 \circ \beta = \iota_A \circ \beta.
\]
Since $[N]_{J(X)}$ is surjective,
\[
\alpha'(E^r) = ([N]_{J(X)} \circ \alpha')(E^r) = \iota_A(\beta(E^r)) \subseteq A = \alpha(E^r).
\]
Exchanging the roles of $\alpha$ and $\alpha'$, we get that $\alpha(E^r) \subseteq \alpha'(E^r)$, so indeed $\alpha'(E^r) = A$.

For any two abelian $K$-varieties $A_1, A_2$, let $\rho_{A_1, A_2}$ denote the smallest positive integer $d$ such that there is a $K$-isogeny $A_1 \to A_2$ of degree $d$.
Let $\mcalF_A = \iota_A^* \mcalL_{J(X)}$, then $(A, \mcalF_A)$ is a polarized abelian variety by part \eqref{item:prop:basics, finite ker} of Proposition \ref{prop:basics}.
Let $\mcalO = \End_K(E)$, this is either $\mbbZ$ or an order in an imaginary quadratic number field, thus $\mcalO$ is an integral domain.
If $R$ is an integral domain with fraction field $F$, and $M$ is an $R$-module, then we define the rank of $M$ over $R$ as $\rank_R(M) := \dim_F(M \otimes_R F)$.
Let $\Omega = \Hom_K(J(X),E)$, then the Rosati involution defines an isometry $\Omega \cong \Hom_K(E,J(X))$.
Note that Rosati involution is a conjugate-linear isomorphism of $\mcalO$-modules.
Furthermore, by Lemma \ref{lemma:|gf| and |f|} we have a canonical isometry $\Hom_K(E,J(X)) \cong \Hom_K(E, A)$, where the Rosati metric on $\Hom_K(E,A)$ is defined using the polarization $\mcalF_A$ on $A$.
There are isogenies $\alpha \colon E^r \to A$, $\beta \colon A \to E^r$ and a positive integer $N$ such that $\beta \circ \alpha = [N]_{E^r}$ and $\alpha \circ \beta = [N]_A$.
These induce $\mcalO$-module homomorphisms
\[
\Phi_{\alpha} = (\alpha \circ \cdot) \colon \Hom_K(E,E^r) \to \Hom_K(E, A)
\]
and
\[
\Phi_{\beta} = (\beta \circ \cdot) \colon \Hom_K(E,A) \to \Hom_K(E,E^r).
\]
Then $\Phi_{\alpha} \circ \Phi_{\beta} = N \identity_{\Hom_K(E,A)}$ and $\Phi_{\beta} \circ \Phi_{\alpha} = N \identity_{\Hom_K(E,E^r)}$, therefore
\[
\rank_{\mcalO}(\Omega) = \rank_{\mcalO}(\Hom_K(E,A)) = \rank_{\mcalO}(\Hom_K(E,E^r)) = \rank_{\mcalO}(\mcalO^r) = r.
\]
We assume that there is a non-constant $K$-morphism from $X$ to $Y$, so $\Omega \neq \{0\}$, hence $r \ge 1$.

Note that $\Omega$ is a finitely generated, torsion free $\mcalO$-module of rank $r$.
Define for such an $\mcalO$-module $\Omega$ the quantity
\[
c(\Omega) = \min \left\{\left|\Omega/\left(\sum_{i=1}^r \mcalO \alpha_i\right)\right| ; \, \alpha_1, \dotsc, \alpha_r \in \Omega \right\}.
\]
One can easily check that there is an $\mcalO$-submodule $\Omega' \subseteq \Omega$ such that $\Omega' \cong \mcalO^r$, and there is an $s \in \Omega \setminus \{0\}$ such that $s \Omega \subseteq \Omega'$.
Then $c(\Omega) \le |\Omega/\Omega'| \le |\mcalO^r/(s \mcalO^r)| = |\mcalO/(s \mcalO)|^r < \infty$.
So $c(\Omega)$ is a positive integer.
If $\mcalO = \mbbZ$, or more generally, if $\mcalO$ is a principal ideal domain, then $\Omega \cong \mcalO^r$, hence $c(\Omega) = 1$.

\begin{proposition} \label{prop:covolume}
If $\mcalO \cong \mbbZ$, then
\[
\vol(\Omega) = 2^{r/2} \sqrt{\chi(A, \mcalF_A) \rho_{E^r, A}},
\]
while if $\mcalO \not\cong \mbbZ$, then
\[
\vol(\Omega) = \frac{1}{c_{\mcalO}(\Omega)} |D(\mcalO)|^{r/2} \chi(A, \mcalF_A) \rho_{E^r, A}.
\]
Furthermore
\begin{align*}
\chi(A, \mcalF_A) \rho_{E^r, A} &= \min \{\chi(E^r, \gamma^* \mcalF_A) ; \, \gamma \in \Hom_K(E^r, A) \textrm{ is a $K$-isogeny} \} \\
&= \min \{\chi(E^r, \gamma^* \mcalL_{J(X)}) ; \, \gamma \in \Hom_K(E^r,J(X)), \, \Ker \gamma \textrm{ is finite}\} \\
&= \sqrt{\rho_{E^r \times B, J(X)} \rho_{E^r, A}} \le \rho_{E^r \times B, J(X)}.
\end{align*}
\end{proposition}
\begin{proof}
We start by proving the last statement.
If $\widetilde{\gamma} \colon E^r \to J(X)$ is a $K$-morphism with finite kernel, then the image of $\widetilde{\gamma}$ is $A$, so $\widetilde{\gamma}$ induces a $K$-isogeny $\gamma \colon E^r \to A$ such that $\widetilde{\gamma} = \iota_A \circ \gamma$.
Conversely, every $K$-isogeny $\gamma$ induces a $K$-morphism $\widetilde{\gamma} = \iota_A \circ \gamma \colon E^r \to J(X)$ with finite kernel.
Here $\gamma^* \mcalL_{J(X)} = \widetilde{\gamma}^* \mcalF_A$, so the penultimate equation is clear.
Using parts \eqref{item:prop:basics, isogeny}, \eqref{item:prop:basics, h^i} and \eqref{item:prop:basics, finite ker} of Proposition \ref{prop:basics}, we get that $\mcalF_A$ and $\gamma^* \mcalF_A$ are ample, $\chi(A, \mcalF_A) > 0$ and $\chi(E^r, \gamma^* \mcalF_A) = (\deg \gamma) \chi(A, \mcalF_A)$.
So $\deg \gamma$ is minimal if and only if $\chi(E^r, \gamma^* \mcalF_A)$ is minimal, therefore
\[
\chi(A, \mcalF_A) \rho_{E^r, A} = \min \{\chi(E^r, \gamma^* \mcalF_A) ; \, \gamma \in \Hom_K(E^r, A) \textrm{ is a $K$-isogeny} \}.
\]
It is easy to see that every $K$-morphism $E^r \times B \to J(X)$ factors through the $K$-isogeny $\varphi \colon A \times B \to J(X)$, $(a,b) \mapsto a+b$.
So
\[
\rho_{E^r \times B, J(X)} = \rho_{E^r, A} \cdot \deg \varphi \ge \rho_{E^r, A}.
\]
Note that \cite[Corollary 5.3.6]{BirLa} implies that $\varphi^* \mcalL_{J(X)} \cong \mcalF_A \boxtimes (\iota_B^* \mcalL_{J(X)})$ on $A \times B$ (here we use the notation $\boxtimes$ of \cite[p.\ 2065]{GauRe}), so using part \eqref{item:prop:basics, isogeny} of Proposition \ref{prop:basics} we obtain
\[
\deg \varphi = \chi(J(X), \mcalL_{J(X)}) \deg \varphi = \chi(A \times B, \varphi^* \mcalL_{J(X)}) = \chi(A, \mcalF_A) \chi(B, \iota_B^* \mcalL_{J(X)}).
\]
Here $\chi(A, \mcalF_A) = \chi(B, \iota_B^* \mcalL_{J(X)})$ (see Corollary 12.1.5 in \cite{BirLa} or Lemma 1.3 in \cite{MW:Periods}), so $\deg \varphi = \chi(A, \mcalF_A)^2$.
Hence
\[
\rho_{E^r \times B, J(X)} \ge \sqrt{\rho_{E^r \times B, J(X)} \rho_{E^r, A}} = \rho_{E^r, A} \sqrt{\deg \varphi} = \chi(A, \mcalF_A) \rho_{E^r, A}.
\]

The isometry $\Omega \cong \Hom_K(E,A)$ shows that $\vol(\Omega) = \vol(\Hom_K(E,A))$.
Note that $\Hom_K(E,A)$ is a torsion-free $\mcalO$-module of rank $r$.
Let $\gamma = (\gamma_1, \dotsc, \gamma_r) \colon E^r \to A$ be a $K$-morphism, where $\gamma_1, \dotsc, \gamma_r \in \Hom_K(E,A)$.
We claim that $\gamma$ is an isogeny if and only if $\gamma_1, \dotsc, \gamma_r$ are $\mcalO$-linearly independent.
There are $K$-isogenies $\alpha \colon E^r \to A$, $\beta \colon A \to E^r$, and a positive integer $N$ such that $\beta \circ \alpha = [N]_{E^r}$ and $\alpha \circ \beta = [N]_A$.
Then $\gamma$ is an isogeny if and only if $\beta \circ \gamma \in \End_K(E^r) \cong \mcalO^{r \times r}$ is an isogeny.
Let $P \in \mcalO^{r \times r}$ be the matrix corresponding to $\beta \circ \gamma$.
Then $\beta \circ \gamma$ is an isogeny if and only if $\det P \neq 0$, which is equivalent to the $\mcalO$-linear independence of $\beta \circ \gamma_1, \dotsc, \beta \circ \gamma_r$.
Using $\alpha \circ \beta = [N]_A$, we see that this is equivalent to the $\mcalO$-linear independence of $\gamma_1, \dotsc, \gamma_r$, hence our claim is true.
Furthermore, it is easy to see that if $\gamma$ is an isogeny, then $\bigoplus_{i=1}^r \mcalO \gamma_i$ is a finite index subgroup of $\Hom_K(E,A)$.
Let $d_{\gamma} = |\Hom_K(E,A)/\bigoplus_{i=1}^r \mcalO \gamma_i|$.

The embedding
\[
(\Hom_K(E,E^r), q_{\mcalL_E, \gamma^* \mcalF_A}) \xrightarrow{\gamma \circ \cdot} (\Hom_K(E,A), q_{\mcalL_E, \mcalF_A})
\]
respects the metrics, and the image has index $d_{\gamma}$, so
\[
\vol(\Hom_K(E,E^r), q_{\mcalL_E, \gamma^* \mcalF_A}) = d_{\gamma} \vol(\Hom_K(E,A), q_{\mcalL_E, \mcalF_A}) = d_{\gamma} \vol(\Omega).
\]
By calculating the left hand side of the above equation, we will show that
\[
d_{\gamma} \vol(\Omega) = 2^{r/2} \sqrt{\chi(E^r, \gamma^* \mcalF_A)}
\]
if $\mcalO \cong \mbbZ$, and
\[
d_{\gamma} \vol(\Omega) = |D(\mcalO)|^{r/2} \chi(E^r, \gamma^* \mcalF_A)
\]
otherwise.
Taking the $\gamma$ which minimizes both sides, we get the statement of the Proposition.

Let $M \in \mcalO^{r \times r}$ denote the matrix corresponding to the polarization $\phi_{\gamma^* \mcalF_A} \in \End_K(E^r)$.
This is a hermitian positive definite matrix with $\det M = \chi(E^r, \gamma^* \mcalF_A)$ by Lemma \ref{lemma:matrix of phi_F}.
Let $x = (x_1, \dotsc, x_r)^T \in \Hom_K(E,E^r) \cong \mcalO^r$, where $x_i \in \mcalO$.
We will show that $\frac{1}{2} |x|_{(\mcalL_E, \gamma^* \mcalF_A)}^2 = x^* M x$, where $x^*$ denotes the conjugate transpose of $x$.
First note that
\[
|x|_{(\mcalL_E, \gamma^* \mcalF_A)} = |x \circ [1]_E|_{(\mcalL_E, \gamma^* \mcalF_A)} = |[1]_E|_{(\mcalL_E, x^* \gamma^* \mcalF_A)}
\]
by Lemma \ref{lemma:|gf| and |f|}.
By Proposition \ref{prop:basics2}, there is a positive integer $k$ such that $\phi_{x^* \gamma^* \mcalF_A} = k \phi_{\mcalL_E}$.
Using the definition of the Rosati norm and Proposition \ref{prop:endoms of elliptic curves}, we get that
\[
\frac{1}{2} |x|^2_{(\mcalL_E, \gamma^* \mcalF_A)} = \frac{1}{2} |[1]_E|^2_{(\mcalL_E, x^* \gamma^* \mcalF_A)} = \frac{k}{2} |[1]_E|^2_{(\mcalL_E, \mcalL_E)} = k.
\]
Using Lemma \ref{lemma:phi_f^*F = f' phi f}, we get that $x^* M x \in \mcalO$ corresponds to
\[
\widehat{x} \circ \phi_{\gamma^* \mcalF_A} \circ x = \phi_{x^* \gamma^* \mcalF_A} = k \phi_{\mcalL_E} \in \End_K(E,E),
\]
so $x^* M x = k = \frac{1}{2} |x|^2_{(\mcalL_E, \gamma^* \mcalF_A)}$.
Hence $|x|^2_{(\mcalL_E, \gamma^* \mcalF_A)} = x^* (2M) x$.
So if $\mcalO \cong \mbbZ$, then on $\Hom_K(E,E^r) \cong \mbbZ^r$ the quadratic form $q_{\mcalL_E, \gamma^* \mcalF}$ is defined by the positive definite matrix $2M \in \mbbZ^{r \times r}$, hence
\[
\vol(\Hom_K(E,E^r), q_{\mcalL_E, \gamma^* \mcalF_A}) = \sqrt{\det(2M)} = \sqrt{2^r \det M} = 2^{r/2} \sqrt{\chi(E^r, \gamma^* \mcalF_A)}.
\]
Suppose that $\mcalO \cong \mbbZ + \mbbZ \tau$.
Let $e_1, \dotsc, e_r$ be the standard $\mcalO$-basis of $\Hom_K(E,E^r) \cong \mcalO^r$.
Then $e_1, \dotsc, e_r, \tau e_1, \dotsc, \tau e_r$ is a $\mbbZ$-basis of $\Hom_K(E,E^r) \cong \mbbZ^{2r}$.
If $z = (z_1, \dotsc, z_{2r})^T \in \mbbZ^{2r}$ and
\[
x = (z_1 + z_{r+1} \tau, \dotsc, z_r + z_{2r} \tau) \in \mcalO^r \cong \Hom_K(E,E^r),
\]
then $\frac{1}{2} |x|^2_{(\mcalL_E, \gamma^* \mcalF_A)} = x^* M x = z^T N z$, where $N = \left( \begin{smallmatrix} \Repart(M) & \Repart(\tau M) \\ \Repart(\overline{\tau} M) & |\tau|^2 \Repart(M) \end{smallmatrix} \right) \in \mbbR^{2r \times 2r}$.
So $|x|^2_{(\mcalL_E, \gamma^* \mcalF_A)} = z^T (2N) z$, hence
\[
\vol(\Hom_K(E,E^r), q_{\mcalL_E, \gamma^* \mcalF_A}) = \sqrt{\det(2N)} = 2^r |\Impart \tau|^r \cdot |\det M|
\]
by Lemma \ref{lemma:matrix lemma}.
Since $|D(\mcalO)| = 4 |\Impart \tau|^2$ and $\det M = \chi(E^r, \gamma^* \mcalF_A) > 0$, we see that
\[
\vol(\Hom_K(E,E^r), q_{\mcalL_E, \gamma^* \mcalF_A}) = |D(\mcalO)|^{r/2} \chi(E^r, \gamma^* \mcalF_A).
\]
\end{proof}

\begin{remark} \label{remark:c(Omega)}
It is a natural question to ask how big $c(\Omega)$ can be.
The Class Index Lemma in \cite{MW:FactEst} gives the upper bound $c(\Omega) \le |D(\mcalO)|^{r/2}$.
\end{remark}

%=======================================================
\section{Finishing the proof of the theorem}
%=======================================================

In Proposition \ref{prop:covolume}, the discriminant $|D(\mcalO)|$ appears.
So to finish the proof, we first need an upper bound for $|D(\mcalO)|$.
Such a bound is given in the following proposition.

\begin{proposition} \label{prop:discr-bound}
Let $K$ be a number field, $E/K$ an elliptic curve with $\mcalO = \End_K(E) \cong \mbbZ + \mbbZ \tau$, and let $D(\mcalO) = -4 (\Impart \tau)^2$.
Then
\[
|D(\mcalO)| \le 500 [K:\mbbQ]^2 \max(h_F(E),1)^2.
\]
\end{proposition}
\begin{proof}
Let $F = \mbbQ(\tau)$ be the fraction field of $\mcalO$, then $\mcalO$ is an order in $\mcalO_F$.
Let $T = \frac{1}{[K:\mbbQ]} \sum_{\sigma \colon K \hookrightarrow \mbbC} \Impart \tau_{\sigma}$, where $E_{\sigma} \cong \mbbC/(\mbbZ + \mbbZ \tau_{\sigma})$, and $\tau_{\sigma}$ is in the closed fundamental domain $R = \{z \in \mbbC; \, \Impart z > 0, \, |\Repart z| \le \frac{1}{2}, \, |z| \ge 1 \}$.
Here $\Impart \tau_{\sigma}$ is determined by $j(\tau_{\sigma}) = j(E_{\mbbC}) = \sigma(j(E))$, where $j(E) \in K$.
We can choose $\tau$ such that $\tau \in R$.
By Theorem 11.1 and Proposition 13.2 of \cite{Cox}, $j(E)$ and $j(\mcalO)$ have the same minimal polynomial over $\mbbQ$.
Thus there is a $\sigma \colon K \hookrightarrow \mbbC$ such that $j(\tau_{\sigma}) = \sigma(j(E)) = j(\mcalO) = j(\tau)$.
Then
\[
\frac{1}{2} \sqrt{|D(\mcalO)|} = \Impart \tau = \Impart \tau_{\sigma} \le [K:\mbbQ] T \le 6.45 [K:\mbbQ] \max\left(h_F(E) + \frac{\log \pi}{2}, 1\right)
\]
by Proposition 3.2 and Remarque 3.3 of \cite{GauRePeriod}.
Here $\max\left(h_F(E) + \frac{\log \pi}{2}, 1\right) \le (1+\frac{\log \pi}{2}) \max(h_F(E),1)$, so statement follows from $(2 \cdot 6.45 \cdot (1+\frac{\log \pi}{2}))^2 < 500$.
\end{proof}

\begin{proof}[Proof of Theorem \ref{thm:Faltings height}]
Let $n = \rank_{\mbbZ} \Omega = r \rank_{\mbbZ} \mcalO \ge 1$.
We have $\lambda(\Omega) \le \sqrt{n} \vol(\Omega)^{1/n}$ and $\lambda(\Omega)^{n-1} \Lambda(\Omega) \le n^{n/2} \vol(\Omega)$ by \cite[Lemme 3.1]{GauRe}.
If $\omega \in \Omega \setminus \{0\}$, then $|\omega|^2 = \Trace(\omega \circ \omega^{\dag})$ is a positive integer (see \cite[Propri\'et\'e 2.2 (1)]{GauRe}), so $|\omega| \ge 1$.
Thus $\lambda(\Omega) \ge 1$, hence $\Lambda(\Omega) \le n^{n/2} \vol(\Omega)$.
Using $R(\Omega) \le \frac{\sqrt{n}}{2} \Lambda(\Omega)$, we get that $R(\Omega) \le \frac{1}{2} n^{(n+1)/2} \vol(\Omega)$.
Substituting in Corollary \ref{cor:mu(X,Y)=1/2 min(|f|)^2}, and using $n \le 2r \le 2g(X)$, we get that
\[
\mu_K(X,Y) \le p(X)^2 \cdot (2g(X))^{2g(X)+1} \vol(\Omega)^2.
\]

In Proposition \ref{prop:covolume}, we have $\rho_{E^r \times B, J(X)} \le \kappa(J(X))$ by \cite[Th\'eor\`eme 1.4.]{GauRe}, and $|D(\mcalO)| \le 500 [K:\mbbQ]^2 \max(h_F(E),1)^2$ by Proposition \ref{prop:discr-bound}.
By \cite[Corollaire 1.5]{GauRe}, $h_F(A) \le h_F(J(X)) + \frac{1}{2} \log \kappa(J(X))$, and there is an isogeny $E^r \to A$ of degree at most $\kappa(A) \le \kappa(J(X))$, so
\[
h_F(E) \le r h_F(E) = h_F(E^r) \le h_F(A) + \frac{1}{2} \log \kappa(J(X)) \le h_F(J(X)) + \log \kappa(J(X)).
\]
It is easy to check using the definition of $\kappa(J(X))$ that
\[
|D(\mcalO)| \le 500 [K:\mbbQ]^2 \max(h_F(J(X)) + \log \kappa(J(X)),1)^2 \le \kappa(J(X))^{\frac{1}{2g(X)}}.
\]
So if $\mcalO \cong \mbbZ$, then $\vol(\Omega)^2 \le 2^{g(X)} \kappa(J(X))$, hence
\[
\mu_K(X,Y)/p(X)^2 \le (2g(X))^{2g(X)+1} \cdot 2^{g(X)} \kappa(J(X)) \le \kappa(J(X))^3.
\]
If $\mcalO \not\cong \mbbZ$, then $\vol(\Omega)^2 \le |D(\mcalO)|^{g(X)} \kappa(J(X))^2 \le \kappa(J(X))^{5/2}$, so
\[
\mu_K(X,Y)/p(X)^2 \le (2g(X))^{2g(X)+1} \cdot \kappa(J(X))^{5/2} \le \kappa(J(X))^3.
\]
\end{proof}

\subsection*{Acknowledgements}
This paper has its origins in the author's Ph.D.\ studies under the supervision of Gisbert W\"ustholz at ETH Z\"urich.
Therefore the author thanks Gisbert W\"ustholz for introducing him to this field, and for all the helpful discussions.

\end{document}